\documentclass[12pt]{article}

\usepackage{amsmath}
\usepackage{amsfonts}
\usepackage{amssymb}
\usepackage{amsthm}
\usepackage{todonotes}

\newtheorem{theorem}{Theorem}
\newtheorem{corollary}{Corollary}
\newtheorem{claim}{Claim}
\newtheorem{example}{Example}
\newtheorem{definition}{Definition}

\title{Signatures of Branched Coverings\footnote{To be submitted for publication in proceedings of 6th European Congress of Mathematics}.}
\author{Yuri Burda, Askold Khovanskii}

\begin{document}

\maketitle

\begin{abstract}

In this paper we deal with branched coverings over the complement to finitely many exceptional points on the Riemann sphere having the property that the local monodromy around each of the branching points is of finite order. To such a covering we assign its \textit{signature}, i.e. the set of its exceptional and branching points together with the orders of local monodromy operators around the branching points. 

What can be said about the monodromy group of a branched covering if its signature is known? It seems at first that the answer is nothing or next to nothing. Indeed, generically it is so. However there is a (small) list of signatures of \textit{elliptic} and \textit{parabolic} types, for which the monodromy group can be described completely, or at least determined up to an abelian factor. This appendix is devoted to investigation of these signatures. For all these signatures (with one exception) the corresponding monodromy groups turn out to be solvable.

Linear differential equations of Fuchs type related to these signatures are solvable in quadratures (in the case of elliptic signatures --- in algebraic functions). A well-known example of this type is provided by Euler differential equations, which can be reduced to linear differential equations with constant coefficients.

The algebraic functions related to all (but one) of these signatures are expressible in radicals. A simple example of this kind is provided by the possibility to express the inverse of a Chebyshev polynomial in radicals. Another example of this kind is provided by functions related to division theorems for the argument of elliptic functions. Such functions play a central role in the work [1] of Ritt.

\end{abstract}

\section{Coverings with a Given Signature}

\subsection{Definitions and Examples}

The mapping $\pi:Y\rightarrow S$ of a connected Riemann surface $Y$ to the Riemann sphere $S$ is called \textit{ admissible}, if the following conditions hold:

1) $\pi(Y)=S\setminus B$, where $B=\{b_{n+1},\dots b_{n+k}\}$ is the \textit{ exceptional} set;

2) $\pi:Y\rightarrow S\setminus B$ is a branched covering with the branching locus $A=\{a_1,\dots,a_n\}$;

3) For $1\leq j\leq n$ the order of the local monodromy operator at the point $a_j$ is a finite number $r_j>1$ (the local monodromy operator at point $x$ is the element of the monodromy group, defined up to conjugation, that corresponds to a small path going around the point $x$). We don't assume anything about the order of local monodromy operators at points $b_j$  (i.e. points $b_j\in B$ can be branching points of infinite order).

\begin{definition}The \textit{signature} of an admissible mapping  $\pi:Y\rightarrow S$ is the triple $(A,\,B,\, R)$, where  $R=\{r_1,\dots,r_n,\infty,\dots,\infty\}$ is the \textit{set of orders}. If $B=\emptyset$, we don't mention $B$ in the signature. We call an admissible mapping with a given signature a \textit{covering with a given signature}.
\end{definition}

We assume that \textit{the inequality $n+k\geq2$ holds for the signature $(A,\,B,\,R)$}. We also assume that \textit{for the signature $(A,\,R)$ with $\# A=2$ and $R=(k,n)$ the equality $k=n$ holds}. If a signature does not satisfy these conditions, then any covering with such signature is either trivial or does not exits.

\begin{example}Consider an algebraic function with the branching locus $A=\{a_1,\dots,a_n\}$. Suppose that the local monodromy operator at the point $a_i\in A$  has order $r_i$. Then the Riemann surface of this function is a covering with signature $(A,R)$, where $R=\{r_1,\dots,r_n\}$.
 \end{example}
 
\begin{example}Consider a linear differential equation of Fuchs type with the set of singular points $A\cup B$, where $A=\{a_1,\dots,a_n\}$, $B=\{b_{n+1}\dots, b_{n+k}\}$. Suppose that the local monodromy operator has a finite order $r_i$ at each of the points $a_i\in A$ and an infinite order at each of the points $b_j\in B$. Then the Riemann surface of a generic solution of this differential equation is a covering with signature $(A,B,R)$, where $R=\{ r_1,\dots,r_n,\infty,\dots,\infty \}$.
\end{example}

We will see below that for all but one of the exceptional signatures the set $A\cup B$ contains two or three points.
\begin{claim} If $\# A\cup B\leq 3$ then up to an automorphism of the sphere $S$, the signature $(A,\,B,\,R)$ is defined by the set of orders $R$.
\end{claim}
\begin{proof} There exists an automorphism of the sphere that takes any given triple of points to any other triple.
\end{proof}

\subsection{Classification}

The covering $\pi:Z\rightarrow S$ with signature $(A,\,B,\,R)$ is called \textit{universal} if: 1) the surface $Z$ is simply-connected, 2) the multiplicity of the mapping $\pi$ at points $c_j\in \pi^{-1}(a_k)$ is $r_k$. 

The universal covering $\pi:Z\rightarrow  S$ with signature $(A,\,B,\,R)$ has the following universal property.

 \begin{theorem} Let $\pi_1: Y\rightarrow  S$ be a covering with signature $(A,\,B,\,R)$ and  $z_0\in Z$, $y_0\in Y$ be points with $\pi (z_0)=\pi_1(y_0)=x_0 \notin A$. Then there exists a mapping $\pi_2:Z\rightarrow Y$ such that $\pi=\pi_1\circ\pi_2$ and $\pi_2 (z_0)=y_0$.
\end{theorem}

\begin{proof} Let $C=\pi^{-1}(A)\subset Z$. Since the surface $Z$ is simply-connected, the fundamental group of the complement $Z\setminus C$ is generated by the curves $\tau_j$ going around the points $c_j\in C$. Suppose $\pi(c_j)=a_k$. By definition the mapping $\pi$ has multiplicity  $r_k$ at point $c_j$. Hence the image of the curve $\tau _j$ under the projection $\pi_1$ goes around the point $a_k$ exactly $r_k$ times. By definition of signature, the lift of the curve $\pi(\gamma)$ to the surface $Y$ based at the point $y_0$ is a closed curve. The theorem follows.
\end{proof}

Let $\pi_1 :Y\rightarrow S$ be a covering with signature $(A,\,B,\,R)$. Fix a point $x_0\in S \setminus (A\cup B)$. A branched covering $\pi :Y\rightarrow S\setminus A$ corresponds to a conjugacy class of subgroups of the fundamental group of the set $S\setminus (A\cup B)$ with base point $x_0$. To the intersection of these subgroups corresponds a branched covering  $\pi_{nor}:Y_{nor}\rightarrow S\setminus A$. This covering will be called the \textit{normalization}  of the original covering.

The following theorem obviously holds.

\begin{theorem} The normalization of a covering with a given signature $(A,B,R)$ is a covering with the same signature and isomorphic monodromy group. If $\pi_{nor}(c_j)=a_k$, then the multiplicity of the mapping $\pi_{nor}$ at point $c_j$ is $r_k$.
\end{theorem}

The following theorem \ref{th:th3} provides an explicit construction of the universal covering with a given signature if some covering with this signature is given.

\begin{theorem}\label{th:th3} Let $\pi_{nor}:Y_{nor}\rightarrow S$  be the normalization of the covering $\pi_1: Y\rightarrow S$ with signature $(A,\,B,\,R)$ and let $\pi :Z\rightarrow Y$ be the universal covering of $Y$. Then $\pi_{nor}\circ \pi: Z\rightarrow S$ is the universal covering with signature $(A,\,B,\,R)$.
\end{theorem}

\begin{proof} By construction the surface $Z$ is simply-connected. If $\pi \circ\pi_{nor}(z)=a_k$, then the multiplicity of the mapping $\pi \circ\pi_{nor}$ at point $z$ is equal to $r_k$. Indeed, the mapping $\pi$ is a local diffeomorphism at point $z$, while the mapping $\pi_{nor}$ has multiplicity $r_k$ at the point $\pi(z)$.
\end{proof}

Theorems 1--\ref{th:th3} provide a way to classify all the coverings with a given signature $(A,\,B,\,R)$ by considering the universal covering with the given signature and its group of deck transformations.

Let $\pi :Z\rightarrow S$ be the universal covering with the signature $(A,\,B,\,R)$. The group $G$ of deck transformations of $\pi$ acts on $Z$. The quotient space of $Z$ by the action of $G$ is isomorphic to $S\setminus B$. The set of orbits on which  $G$ acts freely is isomorphic to  $S\setminus (A\cup B)$. If a point $c\in Z$ gets mapped to the point $a_k\in A$  in the quotient space, then the stabilizer of the point $c$ contains $r_k$ elements.

We say that $H\subset G$ is a \textit{free normal subgroup} of the group $G$ if $H$ acts freely on $Z$ and $H$ is a normal subgroup of $G$. We say that the \textit{subgroup $F\subset G$  is admissible} if the intersection $H=\bigcap F_i$ of all the subgroups $F_i$ conjugate to $F$ is a free normal subgroup of $G$.

\begin{corollary} Any covering with signature $(A,\,B,\,R)$ is isomorphic to a quotient of $Z$ by an admissible subroup $F\subset G$. Conjugate subgroups $F_i$ correspond to equivalent coverings. The monodromy group of the covering is isomorphic to the quotient $G/H$, where $H=\bigcap F_j$. A normal covering with signature $(A,\,B,\,R)$ corresponds to a free normal subgroup $H$. Its group of deck transformations is isomorphic to the monodromy group $G/H$.
\end{corollary}

Admissible mappings can be divided into three natural classes.

\begin{definition} The signature of a covering is \textit{elliptic, parabolic or hyperbolic} if the universal covering $\pi: Z\rightarrow S$ with this signature has total space $Z$ isomorphic to the Riemann sphere, line $\Bbb C^1$ or the open unit disc respectively.
\end{definition}
In $\S \S$ 2--3 we discuss coverings with elliptic and parabolic signatures. Now we turn to a geometric construction of a large class of branched coverings.

\subsection{Coverings and Classical Geometries}
By using the geometry of a sphere, Euclidean and hyperbolic planes one can construct universal coverings with many signatures. In this section we use realizations of each of these geometries on a subset $E$ of the Riemann sphere $\Bbb C^1\cup \{\infty\}$:  the sphere is identified with the set $\Bbb C^1\cup \{\infty\}$ by means of the stereographic projection, the Euclidean plane is identified with the line $\Bbb C^1$ and the hyperbolic plane is identified with its Poincare model in the unit disc $|z|<1$.

We consider polygons in $E$ that may have ``vertices at infinity" lying in $\overline E$. For the plane $\Bbb C$ such a vertex  is the point $\infty$ at which two parallel sides meet. For the hyperbolic plane such vertex is a point on the circle $|z|=1$ at which two neighbouring sides meet. The angle at a vertex at infinity is equal to zero.

Let $E$ be the sphere, plane or hyperbolic plane, and let $\Delta\subset \overline E$ be an $(n+k)$-gon with finite vertices $A'=\{a_1',\dots,a_n'\}$ and vertices at infinity $B'=\{b'_{n+1},\dots, b'_{n+k}\}$. Let $R=(r_1,\dots,r_{n+k})$, where $r_i>1$ are natural numbers for $1\leq i\leq n$ and $r_i=\infty$ for $n<i\leq n+k$.

\begin{definition} The polygon $\Delta\subset \overline E$ has signature $(A',\, B',\,R)$, if for $1\leq i\leq n$ its angle at vertex $a'_i$ is $\pi/r_i$ and for $n< i\leq n+k$ its angle at vertex $b'_i$ is 0.
\end{definition}

It is clear that the signature $(A',\,B',\,R)$ with  $\# A'\cup B' \leq 2$ can be a signature of a polygon only if $R=(k,k)$ or $R= (\infty, \infty)$.  We assume that \textit{when $n+k\leq 2$ this condition on the set $R$ holds}.

\begin{definition}
The \textit{characteristic} of the signature $R=(r_1,\dots,r_{n+k})$ is $$\chi(R)=\sum _{1\leq i\leq n+k}(1-1/r_i).$$
\end{definition}

\begin{definition} We say that the set  $R$ is \textit{elliptic, parabolic or hyperbolic} if $\chi(R)<2$, $\chi(R)=2$ or $\chi(R)>2$ respectively.
\end{definition}

\begin{claim} Suppose that the polygon $\Delta\subset \overline E$ has signature $(A',\, B',\,R)$. The set $R$ is elliptic, parabolic or hyperbolic if and only if $E$ is the sphere, Euclidean plane or the hyperbolic plane respectively.
\end{claim}

\begin{proof} On the sphere the sum of external angles of a polygon is $< 2\pi$, on the plane it is $=2\pi$ and on the hyperbolic plane it is $>2 \pi$.
For  $\Delta$ this sum is equal to  $\pi \sum _{1\leq i\leq n+k}(1-1/k_i)=\pi \chi(R)$.
\end{proof}

\begin{definition} Given a polygon $\Delta\subset \overline E$ with signature $(A',\, B',\,R)$  define $\tilde G_{\Delta}$ to be the group of isometries of the space $E$,  generated by reflections in the sides of the polygon. Define the group $G_{\Delta}$ to be the index two subgroup of the group $\tilde G_{\Delta}$, consisting of orientation preserving isometries.
\end{definition}

The condition on the angles of the polygon guarantees that the images $g(\Delta)$ of the polygon $\Delta$ under the action of the group $\tilde G_{\Delta}$ cover the space $E$ without overlaps. Divide the polygons $g(\Delta), g\in \tilde G_{\Delta}$ into two classes -- white, if $g\in G_{\Delta}$, and black otherwise. Let $g_ {l}$  be the reflection in the side $l$ of the polygon $\Delta$. Define (possibly non-convex) polygon  $\diamondsuit$ as the union of polygons $\Delta$ and $g_{l}(\Delta)$ sharing the side $l$. It can be seen from the construction that the polygon $\diamondsuit$ is a fundamental domain for the action of the group $G_{\Delta}$. The polygon $\diamondsuit$ contains $l$and $l$ is not its side. The transformation $g_{l}$ glues each of the sides $l_j$ of the polygon $\Delta$ with the side $g_{l} (l_j)$. The following claim can be easily verified.

\begin{claim} The stabilizer of the vertex $a'_i\in A'$ under the action of the group $G_{\Delta}$ contains $r_i$ elements. The points of $E$ that do not belong to the orbits of the points $a'_i\in A'$ have trivial stabilizers.
\end{claim}

Consider a Riemann mapping $f$ of the polygon $\Delta\in \overline E$ with signature $(A',\,B',\,R)$ onto the upper half-plane. We introduce the following notations:

$A$ is the set $f(A')$,$\,$  $a_k=f(a'_k)$ for $a'_k\in A'$,

 $B$ is the set $f(B')$, $\,$ $b_j=f(b'_j)$ for $b'_j\in B'$.

\begin{theorem} The mapping $f:\Delta\rightarrow \Bbb C^1\cup \{\infty\}$ can be extended to $E$ and defines a universal branched covering with signature $(A,\,B,\,R)$ over the Riemann sphere. The mapping $f$ realizes the quotient of the space $E$ by the action of the group $G_{\Delta}$.
\end{theorem}
\begin{proof} Follows from Riemann-Schwartz reflection principle.
\end{proof}

\section{Spherical Case}
\subsection{Application of Riemann-Hurwitz formula}
Suppose that a discrete group of automorphisms $G$ acts on the sphere $Z$. Then the group $G$ is finite and the quotient space $Z/G$ is a sphere (since there are non-constant analytic mappings of the sphere to a higher-genus Riemann surface). The quotient mapping $Z\rightarrow Z/G$ defines (up to a composition with an automorphism of the sphere $S$) a universal covering $\pi:Z\rightarrow S$ with elliptic  signature $(A,\,R)$.

\begin{claim}The signature $(A,\,R)$ has an elliptic set $R$.
\end{claim}
\begin{proof}
Let $\# G=N$. Riemann-Hurwitz formula implies $2 =2N - \sum _{a_i\in A}(N-N/r_i)= N (2-\chi(R))$. Hence $\chi(R) <2$.
\end{proof}

We now give names to the following sets: 1) $(k, k)$  -- \textit{the set of a $k$-gon}, 2)~$(2,2, k)$ -- \textit{the set of the dihedron $D_k$}, 3)~$(2,3,3)$ --  \textit{the set of tetrahedron}, 4)~$(2,3,4)$ -- \textit{the set of cube/octahedron}, 5) $(2,3,5)$ -- \textit{ the set of dodecahedron/icosahedron}. The sets 1)-5) are elliptic.

\begin{claim} If the signature $(A,\,R)$ is elliptic, then the set $R$  is among the 5 sets mentioned above.
\end{claim}
\begin{proof} It is enough to find all solutions of the inequality $\chi (R)<2$ satisfying the restrictions imposed on $R$ for $n\leq 2$.
\end{proof}

\subsection {Finite Groups of Rotations of the Sphere}
Consider the following polyhedra in $\Bbb R^3$ having the center of mass at the origin:

1) a pyramid with a regular $n$-gon as its base

2) dihedron with $k$ vertices, or, equivalently, a polyhedron consisting of two pyramids like in 1) joined along their base face

3) regular tetrahedron

4) cube or octahedron

5) dodecahedron or icosahedron.

The symmetry planes of each of these polyhedra cut a net of great circles on the unit sphere. This net divides the sphere into a union of equal spherical polygons $\Delta$ (triangles in cases 2)--5) and digons  in case 1). The stereographic projections of these nets can be found on p.227, figure 3. One sees that the signatures $(A',\,R)$ of polygons $\Delta$ in cases 1)--5) have a set $R$, which is equal to the set having the same name (and the same parameter $k$ in cases 1) and 2)).

Each polyhedron $\Delta$ defines a group $\tilde G_{\Delta}$ of isometries of the unit sphere generated by reflections in its sides, and its index 2 subgroup  $G_{\Delta}$ of orientation-preserving isometries from $\tilde G_{\Delta}$.

\begin{definition} The groups of rotations of the sphere described above are called as follows: 1) the group of the $k$-gon, 2) the group of the dihedron $D_k$, 3) the group of the tetrahedron, 4) the group of cube/octahedron, 5) the group of icosahedron/dodecahedron.
\end{definition}

\begin{claim} A spherical polyhedron with signature  $(A'\,R)$ exists if and only if  $R$ is one of the elliptic sets described above.
\end{claim}

\begin{proof}For one of the directions it is enough to find all the solutions of the inequality $\chi(R)<2$ (having in mind the restrictions imposed on $R$ when $n\leq 2$). For the other direction it is enough to exhibit examples of the spherical polygons. All the examples are given by triangles and dihedrons that appear when the sphere is divided into equal polygons by the symmetry planes of the polyhedra described above (see figure 3).
\end{proof}

\begin{theorem} A finite group of automorphisms of the Riemann sphere with a given signature coincides up to an automorphism of the Riemann sphere with a group of rotations of the sphere with the same name as its signature.
\end{theorem}

\subsection{Coverings with Elliptic Signatures}
Every automorphism of the sphere has fixed points and thus the automorphism group of the sphere doesn't have free normal subroups. Fix an elliptic signature.The universal covering with this signature is the Riemann sphere $Z$ equipped with the deck transformation group $G$, the quotient map $ Z \rightarrow Z/G$ and an isomorphism $Z/G\rightarrow S$. 

The coverings with a given elliptic signature are in one to one correspondence with conjugacy classes of subgroups of $G$ that don't have nontrivial normal subgroups of the group $G$. Each such covering has normalization that is equivalent to the universal covering with the same signature and monodromy group isomorphic to the group  $G$. Thus the monodromy group of a covering with an elliptic signature is determined by its signature.

\subsection{ Equations with an Elliptic Signature}
\begin{theorem} An algebraic function with an elliptic signature and a set of orders not equal to the set $(2,3,5)$ can be represented in radicals. If the set of orders is equal to $(2,3,5)$, then it can be represented by radicals and solutions of equations of degree at most 5.
\end{theorem}

\begin{example} The inverse of the Chebyshev polynomial of degree $n$ has signature $A=\{1,-1,\infty\}$,  $R=(2,2,n)$ of elliptic type (the case of the dihedron $D_n$). This explains why the Chebyshev polynomials are invertible in radicals.
\end{example}

\begin{theorem} A linear differential equation of Fuchs type with elliptic signature and the set of orders different from the set $(2,3,5)$ can be solved in radicals. If the set of orders is $(2,3,5)$, then it can be solved in radicals and solution of algebraic equations of degree at most 5.
\end{theorem}

\section{The Case of the Plane}

\subsection {Discrete Groups of Affine Transformations}

Every automorphism of the complex line is an affine transformation $z\rightarrow az+b$ with $a\neq 0$. The group of affine transformations  has a commutative normal subgroup $\Bbb C$ consisting of translations  with a commutative factor-group $\Bbb C^*$. The group of automorphisms of the line is thus solvable and hence all its discrete subgroups are solvable as well. The affine transformations with no fixed points are precisely the translations.

The discrete groups $G$ of the group of affine transformations of the complex line can be classified up to an affine change of coordinates as having one of the eight types below. The space $\Bbb C^1/G$ for each group $G$, except the groups in case 4), is a sphere or a sphere without one or two  points. The quotient  $\Bbb C^1\to \Bbb C^1/G$ defines in these cases a covering with parabolic signature. For all the groups except the group in case 5), the set $A\cup B$ for these signatures consists of at most three points. Hence in this case the signature is defined up to an automorphism of the sphere by the set of its orders $R$.

We use the following notations: $S_k\subset \Bbb C^*$ is the multiplicative subgroup of order $k$, $\Lambda_2=(1,c)$  is the additive group $\Lambda_2\subset \Bbb C$ generated by the numbers $1$ and $c$, where $c\notin  \Bbb R$ is defined up to the action of the modular group;  the number $\lambda \notin \{0,1,\infty\}$ denotes a number under the equivalence where numbers $\lambda, 1-\lambda, \lambda ^{-1},(1-\lambda)^{-1}, \lambda (1-\lambda)^{-1}, \lambda^{-1} (1-\lambda)$  are equivalent,  $\tau_6$ -- a primitive root of unity of order 6.

The groups  $G$ consist of transformation $x\rightarrow ax+b$, where:

1) $a\in S_k$, $b=0$;$\,\,$ $R=(k,\infty)$;

2) $a=1$, $b\in \Bbb Z$; $\,\,$ $R=(\infty,\infty)$;

3)  $a\in S_2$, $b\in \Bbb Z$; $\,\,$ $R=(2,2,\infty)$;

4) $a=1$, $b\in \Lambda_2=(1,c)$;$\,\,$   $\Bbb C^1/G$ is a curve of genus one;

5) $a\in S_2$, $b\in \Lambda_2=(1,c)$;$\,\,$ signature $A=\{0,1,\infty, \lambda,\},R=(2,2,2,2)$;

6) $a\in S_4$, $b\in \Lambda_2=(1,i)$;$\,\,$ $R=(4,4,2)$;

7) $a\in S_3$, $b\in \Lambda_2=(1, \tau_6)$;$\,\,$ $R=(3,3,3)$;

8) $a\in S_6$, $b\in \Lambda_2=(1, \tau_6)$;$\,\,$ $R=(6, 3, 2)$.

\begin{theorem} A discrete group $G$ of affine transformations is up to an affine change of coordinates one of the groups from the list above. The signature of the coverings related to the action of the group is defined up to an automorphism of the sphere by the data from the list.
\end{theorem}
Below we sketch a proof of this relult. If $G$ does not contain translations and only one point is fixed under transformations $g\in G$, $g\neq e$, then $G$ is of type 1). If $G$ consist of translations only, then $G$ has type 2) or  4). If transformations $g_1, g_2\in G$ have different fixed points, then $g_1g_2g_1^{-1}g_2^{-1}\neq e$ and hence $G$ contains  a discrete subgroup of translations $\Lambda _G\neq G$ and hence is of type 2) or of type 4). If $g(z)= az+b$ and $g\in G$, then the multiplication $z\rightarrow az$ defines an automorphism of the lattice $\Lambda_G$. The group of automorphisms of a lattice is a group $S_k$, having at most two elements linearly independent over $Q$. Hence the order $k$ of group $S_k$ must be 1, 2, 3, 4, 6. This leads to the other remaining cases.

A group of type 4) does not belong to our subject, as $\Bbb C/\Lambda_2$ is a torus rather than a sphere.  A group of type 1) is not interesting for our purposes: it uniformizes functions with sets of orders $(k,\infty)$, among which only the functions with sets of orders $(k,k)$ are interesting to us. These functions has already been considered above. All other groups are interesting to us.

These groups (with the exception of the majority of groups of type 5) can be described geometrically by means of planar polygons.

\subsection { Affine Groups Generated by Reflections}

We call the sets of orders mentioned below as follows: 1) $(\infty, \infty)$  -- \textit{the set of a strip}, 2)~$(2,2, \infty)$ -- \textit{the set of a half-strip}, 3)~$(2,2,4)$ --  \textit{the set of a half of a square}, 4)~$(3,3,3)$ -- \textit{ the set of a regular triangle}, 5)~$(2,3,6)$ -- \textit{ the set of a half of a regular triangle}, 6)~$(2,2,2,2)$ -- the set of a rectangle. All sets mentioned above are parabolic.

\begin{claim} A planar polygon with signature $(A',\,B',\,R)$ exists if and only if  $R$ is one of the sets mentioned above. The polygon is defined uniquely by  $R$ up to affine transformations in all cases but the last one. A rectangle is defined by means of such transformation by the quotient of its side lengths.
\end{claim}

\begin{proof} For the proof it is enough to find all the solutions of the equation $\chi=2$, exhibit examples of the required polygons and classify these polygons up to affine transformations. Here we consider only examples: in case~1) it is a strip between two parallel lines. In case 2) it is the triangle obtained by cutting the strip from the first example by a line perpendicular to its sides. In other cases these are the triangles and quadrilaterals appearing in the names of the cases.
\end{proof}

By comparing the lists in s.~3.1--3.2 we see that the groups of types 2)--3) and 6)--8) are subgroups of index two in the groups generated by reflections in a two- or three-gon with the same set $R$. For a group of type 5) this is so if $\lambda\in \Bbb R$:  in this case the covering is given by the inverse of the elliptic Schwartz-Christoffel integral $\int \frac{dz}{\sqrt{p(z)}}$ with $p(z)=z(z-1)(z-\lambda)$. This integral transforms the upper half-plane into a rectangle.

\subsection { Coverings with Parabolic Signatures}

Let a parabolic signature $(A,\,B,\,R)$ be fixed. The universal covering with this signature consists of the line $\Bbb C^1$ equipped with a discrete group of its transformations $G$,  the factorization mapping $\Bbb C^1 \rightarrow \Bbb C^1/G$ and isomorphism $C^1/G \rightarrow S$. If $\# A\cup B\leq 3$ then the position of the points $A\cup B$ has no significance, as any configuration of at most three points on the sphere can be transformed to any other configuration by an automorphism of the sphere. In this case we know the group $G$ and its geometric description.

Consider the case of signature $A= \{a_1, a_2, a_3, a_4 \},\, R= \{2, 2, 2, 2\}$. If the points of the set $A$ lie on a circle, they can be transformed into points $0,1,\infty,\lambda$ with real $\lambda$. For such points we have described the universal covering above as the inverse of the elliptic Schwartz-Christoffel integral of the form $\int \frac{dz}{p(z)}$, $p(z)=z(z-1)(z-\lambda)$.  If the points of the set  $A$ don't lie on a circle, the universal covering can be described as follows. We can assume that  $\infty \notin A$. In this case the universal covering $I^{-1}:\Bbb C^1\to S$ is given by the inverse of the integral $I=\int \frac{dz}{\sqrt{p(z)}}$ with $p(z)=(z-a_1)(z-a_2)(z-a_3)(z-a_4)$. The groupo of deck transformations of this covering is generated by shifts by the elements of the lattice of periods $\Lambda_2$ of the integral $I$ and by multiplication by $(-1)$. The quotient of $\Bbb C^1$ by the group of translations from $\Lambda_2$ is a torus, which is a two-sheeted branched covering over the sphere with branching points $A$.

We now consider the general case of coverings with parabolic signature. The commutator of the group of all automorphisms of the complex line consists of all the translations. The translations are the only transformations that have no fixed points.

To a given parabolic signature one associates the universal covering with this signature and a group $G$ of automorphisms of the line acting as the group of the deck transformations of the covering. Covering with this signature are in one to one correspondence with conjugacy classes of subgroups of $G$, whose intersection $H$ consists only of translations. The monodromy group of this covering $\pi_1 :Y\rightarrow S$ is isomorphic to the group $G/H$ and is determined by the signature up to a quotient by a subgroup $H$ in the commutator of the group $G$.

\subsection {Equations with Parabolic Signatures}

\begin{theorem} A linear differential equation of Fuchs type with parabolic signature can be solved by quadratures. Its monodromy group is a factor group of a group  $G$ a commutative normal subgroup, where the group $G$ depends only on the signature.
\end{theorem}

\begin{theorem} An algebraic function with parabolic signature is expressible in radicals.  Its monodromy group is a factor group of a group  $G$ depending only on the signature by  a commutative normal subgroup.
\end{theorem}

\begin{example} Coverings with the set of orders $(\infty, \infty)$ are uniformized by a group of type 1). Equations of Euler's type $y^{(n)}+a_1y^{(n-1)}x^{-1}+\dots+a_n yx^{-n}=0$ are of this kind.
\end{example}

\begin{example} Coverings with the set of orders $(2, 2, \infty)$ are uniformized by a group of type 2). Equations of the form $$\sum_{i=0}^n a_i\left((1-x^2)\frac{d^2}{d\,x^2}-x\frac{d}{d\,x}\right)^iy=0$$ have this signature. By means of a change of variables $x=\cos z$ such equation can be reduced to an equation with constant coefficients $\sum_{i=0}^n a_i\frac{d^{2i}y}{d\,z^{2i}}=0$. Hence the solutions of this equation are of the form $$y(x)=\sum_j p_j(\arccos x) \cos(\alpha_j \arccos x)+q_j(\arccos x) \sin(\alpha_j \arccos x),$$ with $p_j,q_j$ --- polynomials. In particular all the (multivalued) Chebyshev functions $f_\alpha$ defined by the property $f_\alpha(\frac{x+x^{-1}})=\frac{x^\alpha+x^{-\alpha}}{2}$ are solutions of such equations. For integer $\alpha$ these are Chebyshev polynomials, for $\alpha=1/n$ with integer $n$ these are the inverses of Chebyshev polynomials.
\end{example}

\begin{example}
If $p_4(z)$ is a fourth degree polynomial with roots $z_1,\ldots,z_4$, then the elliptic integral $y(z)=\int\limits_{z_0}^{z}\frac{d\,z}{\sqrt{p_4(z)}}$ has signature $(z_1,\ldots,z_4;2,2,2,2)$ and it is a solution of the Fuchs type differential equation $y''+\frac{1}{2}\frac{p_4'(z)}{p_4(z)}y'=0$ with the same signature.
\end{example}

\section{ Functions with Non-Hyperbolic Signatures in Other Contexts}

Algebraic functions with elliptic signatures are classical objects. For instance he first part of Klein's book [2]. Algebraic functions with non-hyperbolic signatures play a central role in the works of Ritt on rational mappings of prime degree invertible in radicals (see [1], [3], [4]). The reason for their appearance in these works is as follows.

By a result of Galois, an irreducible equation of prime degree $p$ can be solved in radicals if and only if its Galois group is a subgroup of the metacyclic permutation group $\{x\to a x + b \mod p: a\not\equiv 0 \mod p\}$. A permutation $x\to ax+b \mod p$ splits as a product of $1+\frac{p-1}{n}$ disjoint cycles, where $n$ is the order of the element $a$ in the group $Z_p$ (we use the convention that the order of the identity element is $\infty$). According to Riemann-Hurwitz formula, for a function with such monodromy group the formula $$2=2p-\sum {p-1-\frac{p-1}{n_i}}$$ holds, where $n_i$ are the branching order, or $\infty$, if the branching is of order $p$. In particular the inequality $\sum \frac{1}{\widetilde{n_i}} \geq 2$ on the branching orders $\widetilde{n_i}$ holds. This means that the signature of such rational functions is non-hyperbolic.

In dynamics Lattes maps are studied as examples of rational mappings with exceptional (usually exceptionally simple) dynamics --- these are rational mappings induced by an endomorphism of an elliptic curve (see [5],[6]). These mappings have parabolic signature (but they don't exhaust all the examples of rational mappings with parabolic signatures: to describe all such examples one has to include all the mappings of a sphere to itself induced by a homomorphism between two different elliptic curves). Lattes maps have provided the first examples of rational mappings with Julia set equal to the whole Riemann sphere.

\section{Hyperbolic Case}

Let $R$ be a signature of an algebraic function. If the universal covering with this signature is the Riemann sphere or the complex line,  then the monodromy group of any algebraic function with signature $R$ can be described explicitly: it contains a normal subgroup which is an abelian group with at most two generators, and the quotient by this group is a finite group from a finite list of groups associated with the given signature. In contrast, if the universal covering with signature $R$ is the hyperbolic plane then the monodromy group of an algebraic function with such signature can be arbitrarily complicated as the next theorem shows:

\begin{theorem} Let $R$ be a signature of an algebraic function and let the universal covering with signature $R$ have the hyperbolic plane as its total space. Let $G$ be an arbitrary finite group. There exists a covering with signature $R$ and monodromy group $H$, containing a subgroup $H_1$ which has a normal subgroup such that the quotient of $H_1$ by it is isomorphic to $G$ (i.e. the monodromy group $H$ has a subquotient isomorphic to $G$). 
\end{theorem}

\begin{proof} If $\pi: Y \to S$ is the normalization of the covering associated to an algebraic function with signature $R$, then the universal covering $Z\to S$ with signature $R$ can be obtained as the composition of $\pi$ and the universal (unbranched) covering $Z \to Y$. In particular if  $Z$ is the hyperbolic plane, then  $Y$ is topologically a sphere with at least two handles.

Fix a representation of the group $G$ as a factor group of a free group on $k$ generators. Replace the covering $\pi:Y\to S$ by a covering $\pi_1:Y_1\to S$ where $Y_1$ is an unbranched covering of $Y$ and has topological type of a sphere with at least  $k$ handles. The fundamental group of the surface $Y_1$ admits a homomorphism onto the free group with $k$ generators, and hence onto $G$. Let $\pi_1:Y_2\to Y$ be the unbranched covering associated to the kernel of this homomorphism. Then the composition $\pi\circ \pi_1: Y_2\to S$ is a covering with signature $R$, whose monodromy group contains a subgroup admitting a mapping onto $G$ (more precisely it is the subgroup of permutations of the fiber that correspond to loops in the base space that can be lifted to loops in $Y_1$).

\end{proof}

\end{document}